\documentclass{article}
\usepackage[utf8]{inputenc}
\usepackage[utf8]{inputenc}
\usepackage{amsmath}
\pdfoutput=1
\usepackage{setspace}
\usepackage[toc,page]{appendix}
\usepackage[left=2.8cm, right=2.8cm, bottom=2.05cm, top=2.05cm]{geometry}
\usepackage{hyperref}
\usepackage[raggedright]{titlesec}
\usepackage{cleveref}
\usepackage{amssymb}
\usepackage{footnotebackref}
\usepackage{mathtools}
\usepackage{mathabx,graphicx}
\usepackage{tikz-cd}
\usepackage{amsthm}
\usepackage{enumitem}
\usepackage{eqparbox}

\usepackage{hyperref, cleveref}
\usepackage{graphicx}
\usepackage{graphics}
\usepackage{mathrsfs}
\usepackage{cleveref}
\usepackage{thmtools}
\newcommand{\RNum}[1]{\uppercase\expandafter{\romannumeral #1\relax}}
\usepackage{bbm}
\newlist{steps}{enumerate}{1}
\setlist[steps, 1]{label = Step \arabic*:}
\theoremstyle{plain}
\newtheorem{theorem}{Theorem}[section]
\newtheorem{definition}[theorem]{Definition}

\newtheorem{proposition}[theorem]{Proposition}
\newtheorem{lemma}[theorem]{Lemma}
\newtheorem{remark}[theorem]{Remark}
\newtheorem{corollary}[theorem]{Corollary}
\theoremstyle{definition}
\newtheorem{exmp}[theorem]{Example}

\numberwithin{equation}{section}
\usepackage{rotating}
\usepackage{comment}
\usepackage[nottoc,notlot,notlof]{tocbibind} 

\newtheorem*{theorem*}{Theorem}

\title{A Groupoid Approach to the Riemann Integral (and Path Integral Quantization of the Poisson Sigma Model)}
\author{Joshua Lackman\footnote{jlackman@math.toronto.edu}}
\date{}

\begin{document}

\maketitle
\begin{abstract}
\noindent We use groupoids and the van Est map to define Riemann sums on compact manifolds (with boundary), in a coordinate-free way. These Riemann sums converge to the usual integral after taking a limit over all triangulations of the manifold. We show that the van Est map determines the $n$-jet of antisymmetric $n$-cochains. We discuss using this Riemann sum construction to put the Poisson sigma model on a lattice.
\end{abstract}
\tableofcontents
\section{Introduction}
The Riemann integral of a function defined on a rectangular subset of $\mathbb{R}^n$ is defined by using limits of Riemann sums,  whereas the integral of a differential form on an oriented compact manifold is defined by choosing local coordinates and adding up the local integrals using a partition of unity. In this paper we are going to show how we can define Riemann sums associated to differential forms on compact manifolds (with boundary) without choosing local coordinates or a partition of unity. On rectangular subsets of $\mathbb{R}^n$ we recover the usual construction of the Riemann integral. 
\\\\Our approach to the Riemann integral has advantages when studying path integrals originating in the Poisson sigma model (see \cite{strobl}), in particular ones which are related to Kontsevich's solution of the deformation quantization problem of Poisson manifolds (see \cite{kontsevich}). This is because our construction of the Riemann integral naturally leads to a lattice formulation of the aforementioned path integrals, generalizing the construction from usual quantum mechanics. We will briefly discuss this idea in this paper, and we will expand on it in upcoming work — for an in depth precursor, see \cite{Lackman3}.
\\\vspace{0.1cm}\\The construction of the Riemann sums is as indicated in the following theorem, which is the main result of this paper: 
\begin{theorem}[Main]\label{theorem}
Let $X$ be an oriented $n$-dimensional compact manifold (possibly with boundary), and let $\omega$ be an $n$-form. We define the Riemann integral by completing the following steps:
\begin{enumerate}
    \item Antidifferentiate\footnote{This is antidifferentiation in the sense of the van Est map, see \cref{anti}.} $\omega$ to a normalized, $S_n$-antisymmetric $n$-cochain $\Omega$ defined on a neighborhood of the identity bisection inside $\mathbf{B}^n \textup{Pair}(X)=X^{n+1}\,.$ 
    \item Choose an oriented triangulation $\Delta_X$ of $X$ (ie. choose an ordering of the vertices in the simplicial complex)\footnote{This ordering should be chosen so that it determine a class in $H_n(X,\partial X,\mathbb{Z})$ which agrees with the orientation of $X\,.$} and form the associated simplicial set, denoted $X_{\Delta}\,.$ Due to the fact that there is a unique arrow between any two objects in $\textup{Pair}(X),$ we get a natural embedding $\iota:X_{\Delta}\xhookrightarrow{} \textup{Pair}(X)\,.$
    \item We can pull back $\Omega$ to get a cochain $\iota^*\Omega$ on $X_{\Delta}\,.$ The Riemann sum $S_{\Delta}(\omega)$ is then defined to be the sum of $\iota^*\Omega$ over all $n$-simplices in $X_{\Delta}:$
   
    \begin{equation}\label{Riemann sum}
        S_{\Delta}(\omega)=\sum_{n-\textup{simplices } \Delta^n}\iota^*\Omega(\Delta^n)\,.
    \end{equation}
    \item The Riemann integral is then equal to the limit of \ref{Riemann sum} over triangulations of $X$ (after picking any smooth triangulation, we obtain a directed set ordered by linear subdivision, eg. barycentric subdivision):
    \begin{equation}
        \int_X \omega=\lim\limits_{{\Delta}_X} S_{\Delta}(\omega)\,.
    \end{equation}
\end{enumerate}
This construction of the Riemann integral agrees with the usual integral.
\end{theorem}
\vspace{0.1cm}
\begin{remark}
This construction is fully general in the sense that, for a manifold, the choice of partition of unity subordinate to a cover by coordinate charts is enough data to reconstruct an $n$-cochain with the desired properties — take the open neighborhood of the identity bisection in $\mathbf{B}^n \textup{Pair}(X)=X^{n+1}$ to be all $(x_0,\dots,x_n)\in X^{n+1}$ such that $x_0,\ldots,x_n$ are contained in at least one of the coordinate charts.  Use the partition of unity induced on the product to glue together the local antiderivatives of $\omega\,.$
\end{remark}
To prove \Cref{theorem}, the idea is the following: consider a Lie groupoid $G\rightrightarrows X\,,$ and consider a normalized $n$-cochain, ie. a normalized map 
\\
\begin{equation}
\Omega:\underbrace{G\sideset{_t}{_{s}}{\mathop{\times}}   \cdots\sideset{_t}{_{s}}{\mathop{\times}} G}_{n \text{ times}}\to\mathbb{C}\,.
\end{equation}
Let $x\in X\,.$ Denote by $\Omega_x$ the restriction of $\Omega$ to the submanifold of $n$-composable arrows given by 
\begin{equation}
    \{(g_1,\ldots,g_n)\in G\sideset{_t}{_{s}}{\mathop{\times}}   \cdots\sideset{_t}{_{s}}{\mathop{\times}} G:s(g_1)=x\}\,.
    \end{equation}
We want to compute the asymptotic expansion of $\Omega_x$ at the point $(x,\ldots,x)\in G\sideset{_t}{_{s}}{\mathop{\times}}   \cdots\sideset{_t}{_{s}}{\mathop{\times}} G\,.$ To do this, we pick local coordinates $(y^1,\cdots,y^m)$ on the source fiber $s^{-1}(x)$ in a neighborhood containing $x\,,$ whose coordinate we denote by $(x^1,\ldots,x^m)\,.$ This induces local coordinates on $\{G\sideset{_t}{_{s}}{\mathop{\times}}   \cdots\sideset{_t}{_{s}}{\mathop{\times}} G:s(g_1)=x\}\,,$ via the diffeomorphism
\begin{align}
    &\underbrace{G\sideset{_t}{_{s}}{\mathop{\times}} \cdots\sideset{_t}{_{s}}{\mathop{\times}} G}_{n \text{ times}}\to \underbrace{G\sideset{_s}{_{s}}{\mathop{\times}} \cdots\sideset{_s}{_{s}}{\mathop{\times}} G}_{n \text{ times}}\,,
    \\&(g_1,g_2,\ldots,g_n)\mapsto(g_1,g_1g_2,\ldots,g_1g_2\cdots g_n)\,.
\end{align}
Denote these coordinates by $(y_1^1,\ldots,y_1^m,\ldots,y_n^1,\ldots,y_n^m)\,.$ Now in these coordinates, the leading terms in the asymptotic expansion of $\Omega_x$ at the identity bisection are given by the van Est map:    
\begin{lemma}\label{asymptotics}
Let $\Omega$ be a normalized, $S_n$-antisymmetric $n$-cochain on $G\rightrightarrows X\,.$ Then up to order $n$ (in the coordinates and notation defined above), the asymptotic expansion of $\Omega_x$ at the point $(x,\ldots,x)$  is given by 
\begin{equation} 
\sum_{i_1<\cdots<i_n}\frac{1}{n!}VE(\Omega)(\partial_{y^{i_1}},\ldots,\partial_{y^{i_n}})\textup{Vol}_{\Delta}(x^{i_1},\ldots,x^{i_n},y_1^{i_1},\ldots,y_1^{i_n},\ldots,y_n^{i_1},\ldots,y_n^{i_n})\,,
\end{equation}
where $i_1,\ldots,i_n\in \{1,\ldots,m\},$ and $\textup{Vol}_{\Delta}(x^{i_1},\ldots,x^{i_n},y_1^{i_1},\ldots,y_1^{i_n},\ldots,y_n^{i_1},\ldots,y_n^{i_n})$ is the signed volume spanned by the vectors 
\begin{equation}
    (y_1^{i_1}-x^{i_1},\ldots,y_1^{i_n}-x^{i_n}),\ldots,(y_n^{i_1}-x^{i_1},\ldots,y_n^{i_n}-x^{i_n})\in\mathbb{R}^n,
\end{equation}
as given by the determinant. Here, $\partial_{y^{j}}\in\pi^{-1}(x)\subset\mathfrak{g}$ is the coordinate vector field at the point $x\in X\,,$ where $\pi:\mathfrak{g}\to X$ is the projection.  
\end{lemma}
\vspace{0.1cm}
In other words, the van Est map determines the $n$-jet of $\Omega_x$ at the identity bisection. 
As a corollary, we get the following:
\begin{corollary}\label{asymptoticsn}
Suppose that the source fibers of $G\rightrightarrows X$ are $n$-dimensional and that $\Omega$ is a normalized, $S_n$-antisymetric $n$-cochain. Then up to order $n\,,$ the asymptotic expansion of $\Omega_x$ at $(x,\ldots,x)$ is given by 
\begin{equation}
VE(\Omega)(\partial_{y^1},\ldots,\partial_{y^n})\textup{Vol}_{\Delta}(x^{1},\ldots,x^{n},y_1^{1},\ldots,y_1^{n},\ldots,y_n^{1},\ldots,y_n^{n})\,.
\end{equation}
 Here, $\partial_{y^{j}}\in\pi^{-1}(x)\subset\mathfrak{g}$ is the coordinate vector field at the point $x\in X\,,$ where $\pi:\mathfrak{g}\to X$ is the projection.   
\end{corollary}
\vspace{0.1cm}
\Cref{theorem} follows quickly from this corollary. 
\\
\begin{remark}One can formulate a simplicial version of Stokes' theorem using antiysymmetric cochains and it is much simpler than the classical Stokes' theorem, see eg. \cite{kkk} (this is surely related to discrete exterior calculus, eg. \cite{hirani}). \Cref{theorem} provides a bridge between integration on simplicial complexes and classical integration.
\end{remark}
\subsection{Path Integrals and the Poisson Sigma Model}
The Poisson sigma model\footnote{For nonexperts, it may be best to read some of the other sections of this paper before reading this one.} (see \cite{strobl}) is a two dimensional quantum field theory with target space a Poisson manifold, which includes quantum mechanics as a special case (which we will expound on below). Specializing to the case of symplectic manifolds $(M,\omega)$ and letting $X$ be a two dimensional manifold, one has to compute path integrals of the form
\vspace{0.1cm}\\
\begin{equation}\label{exact}
    \int_{\{\phi:X\to M\}}F(\phi)e^{\frac{i}{\hbar}\int_X\phi^*\omega}D\phi\,,
\end{equation}
\vspace{0.1cm}\\
where $F$ is some observable. We want to define this nonperturbatively; one way of doing this is by using the following lattice approach (compare with the usual lattice approach to defining path integrals when the target space is linear, eg. \cite{fradkin}):
\vspace{0.1cm}
\begin{enumerate}
    \item Triangulate $X$ and form the corresponding simplicial set $X_{\Delta}\,,$
    \item Integrate $\omega$ to a $2$-cocycle $\Omega$ on $\textup{Pair}(M)$ (or on just a neighborhood of the identity bisection),
    \item The space of maps $X_{\Delta}\to \textup{Pair}(M)$ is finite dimensional, and we can approximate \ref{exact} by
    \vspace{0.1cm}
    \begin{equation}\label{approx}
        \int_{\{\phi_{\Delta}:X_{\Delta}\to \textup{Pair}(M)\}}F_{\Delta}(\phi_{\Delta})\exp{\bigg[\frac{i}{\hbar}\sum_{X_{\Delta}}\phi_{\Delta}^*\Omega\bigg]}D\phi_{\Delta}\,,
    \end{equation}
    \vspace{0.1cm}
where $F_{\Delta}$ is a suitable approximation to $F$ and $\{\phi_{\Delta}\}$ are just maps of simplicial spaces.
\item We can then define \ref{exact} to be the limit of \ref{approx} over all triangulations. 
\end{enumerate}
\vspace{0.2cm}In order for this definition to be justified we (in particular) need \Cref{theorem} and \Cref{asymptotics} to be true. In the theory of geometric quantization of Poisson manifolds, step 2 constitutes (part of) the geometric quantization data, which we are using in a different way than usual (see \cite{eli}, \cite{weinstein1}). These steps all generalize to Poisson manifolds, where the pair groupoid is replaced by the symplectic groupoid. 
\\\\A particularly interesting special case of this is when $X$ is a disk with three marked points on the boundary, denoted $\{0,1,\infty\}\,.$ The most important observables in this case are of the form $F(\phi)=f(\phi(1))g(\phi(0))\delta_{m}(\phi(\infty))\,,$ where $f,g$ are smooth functions on $M$ and $m\in M\,.$ In this case, we get that \ref{exact} is equal to (see \cite{bon} as well as footnote 3 in \cite{catt})
\vspace{0.1cm}
\begin{align}\label{star}
    (f\star g)(m)=\int_{\begin{subarray}{l}\{\phi:\mathbb{D}\to M:\phi(\infty)=m\end{subarray}\}}f(\phi(1))g(\phi(0))e^{\frac{i}{\hbar}\int_{\mathbb{D}}\phi^*\omega}D\phi\,,
\end{align}
\vspace{0.1cm}\\
which is \textit{normalized} so that $1\star 1=1\,.$ The notation $f\star g$ is due to the fact that the perturbative expansion of this path integral around the classical constant solution is supposed to provide a deformation quantization of $M$ (see \cite{catt}). 
\\\\The perturbative formulation of this path integral does not yield operators a Hilbert space, however a nonperturbative definition would since the map $g\mapsto f\star g$ would associate to any $f$ an operator on $L^2(M,\omega^n)\,.$ If $\omega=d\lambda$ then we can rewrite \ref{star} as  
\vspace{0.1cm}
\begin{align}\label{circle}
    (f\star g)(m)=\int_{\begin{subarray}{l}\{\phi:S^1\to M:\phi(\infty)=m\end{subarray}\}}f(\phi(1))g(\phi(0))e^{\frac{i}{\hbar}\int_{S^1}\phi^*\lambda}D\phi\,.
\end{align}
\subsubsection{Quantum Mechanics}
Specializing to the case that $(M,\omega)=(T^*\mathbb{R},dp\wedge dq)\,,$ we get that \ref{circle} is equal to 
\vspace{0.1cm}
\begin{align}\label{qm}
    (f\star g)(p,q)=\int_{\begin{subarray}{l}\{\phi:S^1\to T^*\mathbb{R}:\,\phi(\infty)=(p,q)\end{subarray}\}}f(\phi(1))g(\phi(0))e^{\frac{i}{\hbar}\int_{S^1}\phi^*p\,dq}D\phi\,.
\end{align}
\vspace{0.1cm}\\
We can compute this path integral nonperturbatively using the antisymmetric $1$-cochain on $\textup{Pair}(T^*\mathbb{R})$ given by 
\vspace{0.1cm}
\begin{equation}
    \Omega(p_0,q_0,p_1,q_1)=\frac{1}{2}(p_0+p_1)(q_1-q_0)=\textup{Alt}[p_1(q_1-q_0)]\,.
\end{equation}
This yields
\vspace{0.1cm}
\begin{align}\label{moy}
          &(f\star g)(p,q)=\frac{1}{(4\pi\hbar)^2}\int_{\mathbb{R}^4}f(p'',q'')g(p',q')e^{\frac{i}{2\hbar}[(p''-p)(q-q')-(q''-q)(p-p')]}\,dp''\,dq''\,dp'\,dq'\,.  \end{align}
          \vspace{0.1cm}\\
This is the nonperturbative form of the Moyal product (see \cite{baker}, \cite{zachos}), and the map $g\mapsto f\star g$ provides a nonperturbative form of the  Wigner–Weyl transform. That is, $f\star g$ quantizes the algebra of functions $C^{\infty}(T^*\mathbb{R})\,.$ Note that, the usual approach to obtaining this algebra involves Lagrangian polarizations (and Fourier transforms), which we have not used (see \cite{eli}).
\\\\We will expand on this in upcoming work, however for a first approximation to these ideas, and where we also discuss the relationship between deformation quantization and twisted convolution algebras of higher groupoids, see \cite{Lackman3}.
\\\\One should compare \ref{qm} to the usual phase space path integral formulation of quantum mechanics, given by (see \cite{fradkin}):
\vspace{0.1cm}
\begin{equation}
    \langle q_f,t_f| q_i,t_i\rangle=\int_{\begin{subarray}{l}\{\phi:[t_i,t_f]\to T^*\mathbb{R}:\,\phi(t_i)=q_i,\phi(t_f)=q_f\end{subarray}\}}D\phi\,e^{\frac{i}{\hbar}\int_{t_i}^{t_f}\phi^*(p\,dq)-(\phi^*H)\,dt}\,.
\end{equation}
\vspace{0.1cm}\\
Here, the vertical polarization $T^*\mathbb{R}\to\mathbb{R}$ is implicitly being used. The usual definition of this is essentially a special case of our construction, with the $1$-cochain being
$ \Omega(p_0,q_0,p_1,q_1)=p_1(q_1-q_0)\,,$ ie.
\vspace{0.1cm}\\
\begin{equation}
\lim\limits_{N\to\infty}\int\prod_{n=1}^{N-1}dq_n\prod_{n=1}^{N}\frac{dp_n}{2\pi\hbar}\exp{\bigg[\frac{i}{\hbar}\sum_{n=1}^N p_n(q_n-q_{n-1})-\frac{(t_f-t_i)}{N}H\Big(p_q,\frac{q_n+q_{n-1}}{2}\Big)\bigg]}\,,
\end{equation}
\vspace{0.1cm}\\
where the integral is over $(-\infty,\infty)$ in all variables (of course, in our construction we haven't included a Hamiltonian).
\section{The van Est Map and Antisymmetric Cochains}
The simplest formulation of the van Est map regards it as a map from the cohomology of a Lie groupoid to the cohomology of its Lie algebroid, and was originally defined for Lie groups in \cite{est}. One of its early applications is attributed to Cartan, who used the van Est map to prove Lie's third theorem (see \cite{est2}). Weinstein and Xu later generalized it to Lie groupoids in \cite{weinstein1}, with applications to geometric quantization in mind. The van Est map has since been generalized to maps between geometric stacks and with more general sheaves in \cite{Lackman}, \cite{Lackman2}.\footnote{For a discussion on the van Est map on $\infty$-groupoids and its relevance to de Rham's theorem, see remark 3 of \cite{Lackman3}.} Various other authors have worked on it, eg. \cite{BS}, \cite{cabrera}, \cite{salazar}.
\\\\The main result regarding the van Est map is the van Est isomorphism theorem, and in the context of \cite{weinstein1} it was proved by Crainic in \cite{Crainic}. While the most general formulation of the van Est map is a work in progress, we will use the most common formulation.
In this context, $VE$ can be defined at the level of cochains. Readers unfamiliar with some of the terminology may wish to consult the appendices.
\\\\We described the original formulation of the van Est map on groupoids in appendix \ref{van est section}, which depends on the description of the nerve of a groupoid $G\rightrightarrows X$ as 
\begin{equation}\label{first}
    \mathbf{B}^n G=\underbrace{G\sideset{_t}{_{s}}{\mathop{\times}} G \sideset{_t}{_{s}}{\mathop{\times}} \cdots\sideset{_t}{_{s}}{\mathop{\times}} G}_{n \text{ times}}\,,
\end{equation}
However, for our purposes it is easier to work with the following construction of the nerve, where we take the fiber product only with respect to the source map:
\begin{definition}\label{bg}
\begin{equation}\label{second}
    \mathbf{B}^n G=\underbrace{G\sideset{_s}{_{s}}{\mathop{\times}} G \sideset{_s}{_{s}}{\mathop{\times}} \cdots\sideset{_s}{_{s}}{\mathop{\times}} G}_{n \text{ times}}\,.
    \end{equation}
We will denote this $\mathbf{B}^n G$ by $G^{(n)}\,.$
\end{definition}
One reason why describing $\mathbf{B}^{\bullet}G$ this way is useful is due to the following: consider the common source map
\\
\begin{equation}
    \underbrace{G\sideset{_s}{_{s}}{\mathop{\times}} \cdots\sideset{_s}{_{s}}{\mathop{\times}} G}_{n \text{ times}}\xrightarrow[]{s}X\,.
\end{equation}
This map has a section $i:X\xhookrightarrow{} G^{(n)}\,, x\mapsto (x,\ldots,x),$ and $i^*\textup{ker}(s_*)=\mathfrak{g}^{\oplus n}\,.$ 
\\\\The two descriptions of the functor $\mathbf{B}^{\bullet}$ are naturally isomorphic. The isomorphism from \ref{first} to \ref{second} is given by
\\
\begin{equation}
    (g_1,g_2,\ldots,g_n)\mapsto (g_1,g_1g_2,\ldots,g_1g_2\cdots g_n)\,.
\end{equation}
\\
Abstractly, we can think of $\mathbf{B}^{\bullet}$ as follows: by the Yoneda lemma we have that
$\mathbf{B}^nG \cong \textup{hom}(\Delta^n,G)\,,$ however since $G$ is a groupoid we can enhance this identification to\footnote{Implicitly, we are using the universal property of localization.}
\begin{equation}\label{pair}
    \mathbf{B}^{n}G\cong\textup{hom}(\textup{Pair}(\{0,\cdots,n\}),G)\,.
\end{equation}
The isomorphism \ref{pair} $\to$ \ref{first} is given by 
\begin{equation}
    f\mapsto (f(0,1),\ldots,f(n-1,n))\,,
\end{equation}
whereas the isomorphism  \ref{pair} $\to$ \ref{second} is given by \begin{equation}
    f \mapsto (f(0,1),\ldots,f(0,n))\,.
\end{equation}
\\
The identification in \ref{pair} makes it clear that $ \mathbf{B}^n G$ comes equipped with an action of $S_{n+1}\,.$ This is because there is an action of $S_{n+1}$ on $\textup{Pair}(\{0,\ldots,n\})\,,$ given by \begin{equation}
\sigma(i,j)=(\sigma(i),\sigma(j))\,,
\end{equation}
for $\sigma\in S_{n+1}\,.$ Hence $\mathbf{B}^{n}G$ inherits an action, ie. for $f\in \textup{hom}(\textup{Pair}(\{0,\cdots,n\}),G)\,,$ \begin{equation}
(\sigma\cdot f)(i,j)=f(\sigma(i),\sigma(j))\,.
\end{equation}
We now describe this action on $G^{(n)}$ (see the appendix of \cite{yag} for an alternative construction of this action on groups):
\begin{definition}
For $\sigma\in S_{n+1}\,,$ we let
\begin{equation}
    \sigma\cdot(g_1,\ldots,g_n):=(g^{-1}_{\sigma^{-1}(0)}g_{\sigma(1)},\ldots,g^{-1}_{\sigma^{-1}(0)}g_{\sigma(n)})\,,
\end{equation}
where $g_0:=\text{id}_{s(g_1)}\,.$
\end{definition}
The subgroup $S_n\subset S_{n+1}$ fixing $0$ acts on $G^{(n)}$ in the expected manner, ie. by permuting components. Now since we have an action of $S_{n+1}$ on $G^{(n)},$ we get an action of $S_{n+1}$ on $n$-cochains $\Omega:G^{(n)}\to\mathbb{C}$ in the usual way, ie.
\\\begin{equation}
(\sigma\cdot \Omega)(g_1,\ldots,g_n)=\Omega(\sigma\cdot(g_1,\ldots,g_n))\,.
\end{equation}
\begin{definition}
An $n$-cochain $\Omega$ on a groupoid $G\rightrightarrows X$ is said to be antisymmetric if $\sigma\cdot\Omega=\textup{sgn}\,(\sigma)\Omega\,.$ If this equality hold only if $\sigma\in S_n$ then $\Omega$ is said to be $S_{n}$-antisymmetric.
\end{definition}
\begin{exmp}\label{ones}
Let $\Omega$ be a $1$-cochain. Then $\Omega$ is antisymmetric if and only if $\Omega(g_1^{-1})=-\Omega(g_1)\,.$
\end{exmp}
As usual, there is an antisymmetrization map:
\begin{definition}\label{alt}
Denote by $\textup{Alt}$ the antisymmetrization map on $n$-cochains with respect to $S_{n+1}\,,$ ie. for an $n$-cochain $\Omega\,,$
\begin{equation}
    \textup{Alt}(\Omega)=\frac{1}{(n+1)!}\sum_{\sigma\in S_{n+1}}\sigma\cdot\Omega\;.
\end{equation}
Similarly, denote by $\textup{Alt}_{n}$ the antisymmetrization map on $n$-cochains with respect to $S_{n}\,.$
\end{definition}
\begin{exmp}
Let $\Omega$ be a $1$-cochain. Then 
\begin{equation}
    \textup{Alt}(\Omega)(g_1)=\frac{\Omega(g_1)-\Omega(g_1^{-1})}{2}\,.
\end{equation}
\end{exmp}
\vspace{0.1cm}\begin{definition}(see eg. \cite{salazar})
An $n$-cochain $\Omega$ is said to be normalized if $\Omega(g_1,\ldots,g_n)=0$ whenever there exists a $j\in\{1,\ldots,n\}$ such that $g_j$ is an identity morphism.
\end{definition}
\vspace{0.2cm}
Antisymmetric cochains better approximate $n$-forms than $S_{n}$-antisymmetric cochains do, which is largely due to the following:
\begin{lemma}
Antisymmetric cochains are normalized.
\begin{proof}
Let $\Omega$ be an antisymmetric $n$-cochain, $n\ge 1\,.$ Let $(0,i)\in S_{n+1}$ be the transpositon switching $0$ with $i\,.$ 
Consider 
\begin{equation}
    (g_1,\ldots,g_i,\textup{id}_{s(g_{i})},\ldots, g_n)\in G^{(n)}\,.
    \end{equation}
    \\
This is fixed by the action of $(0,i)\,,$ which means that 
\begin{equation}
    \Omega(g_1,\ldots,g_i,\text{id}_{s(g_{i})},\ldots ,g_n)=-\Omega(g_1,\ldots,g_i,\text{id}_{s(g_{i})},\ldots ,g_n)\,,
\end{equation}
which proves the result.
\end{proof}
\end{lemma}
\vspace{0.2cm}
This next result is easiest proved by using the construction of the van Est map given in the next section:
\begin{proposition}
Let $\Omega$ be a normalized cochain. Then $\textup{VE}(\Omega)=\textup{VE}(\textup{Alt}(\Omega))\,.$ 
\end{proposition}
Therefore, in the context of normalized cochains and the van Est map, we can assume without loss of generality that any cochains are antisymmetric.
\\
\begin{exmp}
Let $X$ be a manifold and let $G\rightrightarrows X$ be the pair groupoid. Then $G^{(n)}=X^{n+1},$ and $S_{n+1}$ acts by permutations, while $S_n$ acts by permutations fixing the first component. An $n$-cochain is normalized if $\Omega(x_0,\ldots,x_n)=0$ whenever there exists an $i\in \{1,\ldots, n\}$ such that $x_i=x_0\,.$
\end{exmp}
\vspace{0.1cm}
\begin{remark}
For a vector space $V\,,$ the corresponding action of $S_{n+1}$ on $V^{\oplus n}$ is given by
\begin{equation}
    \sigma\cdot(v_1,\ldots,v_n)=(-v_{\sigma^{-1}(0)}+v_{\sigma(1)},\ldots,-v_{\sigma^{-1}(0)}+v_{\sigma(n)})\,,
    \end{equation}
where $v_0:=0\,.$ 
A multilinear map on a vector space is antisymmetric with respect to $S_{n+1}$ if and only if it is antisymmetric with respect to the subgroup $S_n\subset S_{n+1}$ fixing $0\,.$ However, for groupoids this isn't the case, as example \ref{ones} shows.
\end{remark}
\subsection{The van Est map}
The advantage of describing the van Est map using the construction of the nerve given in \cref{bg} is that applying the van Est map manifestly results in an $n$-form (ie. we don't need to choose local extensions).
\\\\ We will implicitly use the following fact, which we mentioned in the previous subsection: consider the common source map
\\
\begin{equation}
    G^{(n)}\xrightarrow[]{s}X\,,\;(g_1,\ldots,g_n)\mapsto s(g_1)\,.
\end{equation}
\\
This map has a section $i:X\xhookrightarrow{} G^{(n)},$ given by $x\mapsto (x,\ldots,x)\,,$ and $i^*\textup{ker}(s_*)=\mathfrak{g}^{\oplus n}\,.$ Now, given an $n$-cochain $\Omega=\Omega(g_1,\cdots,g_n),$ differentiating in the $j\textup{th}$ component of $\Omega$ refers to differentating the map \begin{equation}
    g_j\mapsto\Omega(g_1,\ldots,g_j,\ldots,g_n),
    \end{equation}
with $g_i$ fixed for $i\ne j\,.$ Therefore, given any vector $V\in\mathfrak{g},$ we can use $V$ to differentiate $\Omega$ in any component.
\begin{definition}
Let $\Omega$ be an $n$-cochain on $G\rightrightarrows X\,.$ Then we define
\begin{equation}
    VE(\Omega)(V_1,\ldots,V_n)=n!\,V_1\cdots V_n\,\textup{Alt}_n(\Omega)
\end{equation}
where $V_1,\ldots,V_n\in\mathfrak{g}$ are vectors at the same point in $X\,,$ where $V_j$ differentiates $\textup{Alt}_n(\Omega)$ in the $j\textup{th}$ component (see \cref{alt} for the definition of $\textup{Alt}_n$).
\end{definition}
\begin{proposition}
This definition agrees with definition given appendix \ref{van est section} on normalized cochains, ie. for
\begin{align}
  \nonumber  &f:\underbrace{G\sideset{_t}{_{s}}{\mathop{\times}} G \sideset{_t}{_{s}}{\mathop{\times}} \cdots\sideset{_t}{_{s}}{\mathop{\times}} G}_{n \text{ times}}\to \underbrace{G\sideset{_s}{_{s}}{\mathop{\times}} G \sideset{_s}{_{s}}{\mathop{\times}} \cdots\sideset{_s}{_{s}}{\mathop{\times}} G}_{n \text{ times}}\,, 
  \\&f(g_1,g_2,\ldots,g_n)=(g_1,g_1g_2,\ldots,g_1\cdots g_n)\,,
\end{align}
such that $\Omega$ is normalized, we have that $VE(f^*\Omega)=VE(\Omega)\,.$
\end{proposition}
\begin{proof}(sketch)
Looking at the definition in appendix \ref{van est section}, we only need to extend the vectors to local sections within the corresponding orbit, so we may assume the groupoid is transitive. Since the computation is local, we may assume the groupoid is of the form $\textup{Pair}(X)\times H\rightrightarrows X\,,$ where $H$ is a Lie group.\footnote{Transitive groupoids (ie. groupoids where all objects are isomorphic) are Atiyah groupoids of principal bundles, and the local triviality of principal bundles implies that transitive groupoids are locally of the aforementioned form. See eg. \cite{mac}.} The source and target of $(x,y,h)$ are given by $x,y,$ respectively, and the composition is given by 
\begin{equation}
    (x,y,h)\cdot(y,z,h')=(x,z,hh')\,.
\end{equation}
The result then follows quickly by working in local coordinates, applying the chain rule and using the fact that $\Omega$ is normalized.
\end{proof}
\vspace{0.1cm}
Of course, if $\Omega$ is already $S_n$-antisymmetric, then we get the following:
\begin{proposition}
    If $\Omega$ is $S_n$-antisymmetric, then $VE(\Omega)(V_1,\ldots, V_n)=n!\,V_1\cdots V_n \,\Omega\,.$
\end{proposition}
\vspace{0.1cm}
\begin{exmp}\label{heis}
Consider the antisymmetric Heisenberg cocycle on $(\mathbb{R}^2,+),$ given by $\Omega((a,b),(a',b'))=(ab'-ba')/2,$ with respect to $\mathbf{B}^2G=G\sideset{_t}{_{s}}{\mathop{\times}} G\,.$ We can compute this according to appendix \ref{van est section}, and we get \begin{equation}
    VE(\Omega)=da\wedge db\,.
\end{equation}
\\
Alternatively, with respect to $\mathbf{B}^2G=G\sideset{_s}{_{s}}{\mathop{\times}} G\,,$ we still have that $\Omega((a,b),(a',b'))=(ab'-a'b)/2\,.$ Computing the van Est map, we get
\begin{equation}
    VE(\Omega)(\partial_a,\partial_b)=2\,\partial_a\partial_{b'}\Omega=1\,,
\end{equation}
so that the two computations agree.
\end{exmp}
\begin{exmp}
Consider $\textup{Pair}(\mathbb{R}^m)$ with antisymmetric $n$-cochain $\Omega=\Omega(x_0^1,\ldots,x_0^m,\ldots,x_n^1,\ldots,x_n^m)\,,$ $n\le m\,.$ Let $\partial_{x^{i_1}},\ldots, \partial_{x^{i_n}}\in T_x\mathbb{R}^m\,.$ Then 
\begin{equation}
    VE(\Omega)(\partial_{x^{i_1}},\ldots, \partial_{x^{i_n}})=n!\,\frac{\partial}{\partial_{x_1^{i_1}}}\cdots\frac{\partial}{\partial_{x_n^{i_n}}}\Omega\,.
\end{equation}
\end{exmp}
\begin{exmp}\label{det}Consider $\textup{Pair}(\mathbb{R}^2)$ with antisymmetric $2$-cochain given by \begin{equation}
\Omega(x_0,y_0,x_1,y_1,x_2,y_2)= \frac{1}{2}\begin{vmatrix} x_{1}-x_0 & x_2-x_{0} \\ y_{1}-y_0 & y_{2}-y_0\,. \end{vmatrix}
    \end{equation}
Consider the coordinate vectors $\partial_{x},\partial_{y}$ on $\mathbb{R}^2\,.$
    Then
    \begin{align}
       \nonumber\textup{VE}(\Omega)(\partial_{x},\partial_{y})=
\partial_{x_1}\partial_{y_2}\big[(x_1-x_0)(y_2-y_0)-(x_2-x_0)(y_1-y_0)\big]=1\,,
    \end{align}
    ie. $\textup{VE}(\Omega)=dx\wedge dy\,.$
\end{exmp}
\vspace{1cm}
Moving on, let $x\in X$ and let 
\\
\begin{equation}
    G^{(n)}_x:=\{(g_1,\ldots,g_n)\in G^{(n)}:s(g_1)=x\}\,.
\end{equation}\\
We have that \\
\begin{equation}
  G^{(n)}_x=\underbrace{G^{(1)}_x\times\cdots\times G^{(1)}_x }_{n \text{ times}}\,.
\end{equation}
\\Consider the restriction of $\Omega$ given by
\\
\begin{equation}
    \Omega_x: G^{(n)}_x\to\mathbb{C}\,,\;\,(g_1,\ldots,g_n)\mapsto \Omega(g_1,\ldots,g_n)\,.
\end{equation}\\
Choose local coordinates $(y^1,\ldots,y^m)$ on $s^{-1}(x)$ in a neighborhood of $s(x)\,.$ This induces coordinates on $G^{(n)}_x\,,$ denoted $(y_1^1,\ldots,y_1^m,\ldots,y_n^1,\ldots,y_n^m)\,.$ Let $s(x)=(x^1,\ldots, x^m)$ in coordinates.
\\\\The normalization condition implies that, when computing the asymptotic expansion of $\Omega_x$ at $(x,\ldots,x)\in G^{(n)}_x\,,$ all of the terms which don't involve differentiation in each of the $n$ components of $\Omega_x$ vanish (in particular, all terms of order $\le n-1$ vanish). \Cref{asymptotics} follows quickly.
\\
\begin{exmp}
Looking back at examples \ref{heis} and \ref{det}, we can use \cref{asymptotics} to immediately read off the van Est map.
\end{exmp}
\begin{remark}
We can put our construction of the van Est map into a more general context as follows: consider a surjective submersion $\pi:Y\to X\,,$ and consider an antisymmetric map $\Omega:Y^{(n)}\to\mathbb{C}\,,$ where $Y^{(n)}$ is the $n$-fold fiber product of $Y\to X$ with itself. Let $d_{\pi}$ denote the foliated exterior derivative on $Y^{(n)}\to X,$ ie. it only differentiates along the leaves of this foliation. Then $d_{\pi}\Omega$ pulls back to a foliated $n$-form on $Y\to X,$ ie. an $n$-form that only takes in vectors tangent to the leaves of the foliation. To compute the van Est map, we are applying this construction to $s:G\to X$ and pulling back the vector bundle (of vectors tangent to the leaves) to $X\,.$
\end{remark}
\section{The Main Result}
Here we will prove \Cref{theorem}. Before doing so, let's clarify step 2 (see \cite{spanier} for related topics):
\\\\One way to think of an orientation of an orientable manifold $X$ (with boundary) is as a generator of $H_n(X,\partial X,\mathbb{Z})\,,$ ie. a fundamental class. We can determine such a generator as follows: choose a triangulation of $X$ and can pick an ordering of the vertices ${v_1,\ldots v_m}$ in such a way that the singular chain determined by the triangulation is closed with respect to the boundary map. This determines a generator of $H_n(X,\partial X,\mathbb{Z})\,.$
\\\\Next we can form the simplicial set $X_{\Delta}\,,$ which in degree $k$ consists of the $k$-dimensional faces of the triangulation (or of the corresponding simplicial complex). Since ${v_1,\ldots v_m}\in X\,,$ we get a natural morphism of simplicial sets $X_{\Delta}\xhookrightarrow{}\textup{Pair}(X)\,,$ eg. if $\{v_i,v_j\}$ is an edge with $i\le j\,,$ then $\{v_i,v_j\}\mapsto (v_i,v_j)\,.$
\subsection{The Riemann Integral on a Closed Interval}\label{one}
Before proving \Cref{theorem} in the general case, we will restrict our attention to one dimensional Riemann integrals of functions defined on $[a,b]\,;$ the general case is similar. 
\\\vspace{0.1cm}\\Let $f:[a,b]\to\mathbb{R}$ be a smooth function and consider the one form $f\,dx\,.$ Any $1$-cochain of the form $\Omega(x,y)=G(x,y)(y-x)$ with $G(x,x)=f(x)$ antidifferentiates $f\,dx\,.$ Furthermore, such a cochain is normalized (the $S_1$-antisymmetry condition is trivial).
\\\\
Two such choices are given by $G=s^*f$ and $G=t^*f\,,$ and these will lead to the left and right Riemann sums, respectively. We can also antisymmetrize and take $G(x,y)=(s^*f+t^*f)/2\,.$ We will take $\Omega(x,y)=f(x)(y-x)\,.$
 Next, choose a triangulation of $[a,b],$ ie. a partition $a=x_0<x_1<\ldots<x_n=b\,.$ The Riemann sum given by \ref{Riemann sum} is given by 
\\\vspace{0.1cm}
\begin{equation}\label{oned}
        S_{\Delta}(f\,dx)=\sum_{i=1}^n f(x_{i-1})(x_i-x_{i-1})\,,
    \end{equation}
    \vspace{0.1cm}\\
Taking the direct limit over all triangulations gives us the Riemann integral $\int_a^b f\,dx\,.$
\\\\\vspace{0.1cm}Now there are other options for our 1-cochain, however any other normalized $1$-cochain $\Omega'$ satisfying $VE(\Omega')=f\,dx$ will differ from $\Omega$ by some $\Omega_0$ such that $VE(\Omega_0)=0\,.$ We will now show that the Riemann integral is independent of $\Omega\,.$ To do this, suppose $VE(\Omega_0)=0$ and that $\Omega_0(x,x)=0$ (so that $\Omega_0$ is normalized). We want to compute $S_{\Delta}(0)$ using $\Omega_0\,,$ which gives
\\\vspace{0.1cm}
\begin{equation}\label{zero}
        S_{\Delta}(0)=\sum_{i=1}^n \Omega_0(x_{i-1},x_i)\,.
    \end{equation}
    \vspace{0.1cm}\\
By Taylor's theorem, we know that 
\\\vspace{0.1cm}
\begin{equation}
    \Omega_0(x,y)=\frac{\partial\,\Omega_0}{\partial y}(x,\xi_{x,y})(y-x)
\end{equation}
\vspace{0.1cm}\\
for some $\xi_{x,y}\in [x,y]\,.$ Therefore, \ref{zero} is equal to 
\begin{equation}
        S_{\Delta}(0)=\sum_{i=1}^n \frac{\partial\,\Omega_0}{\partial y}(x_i,\xi_{x_{i-1},x_i})(x_i-x_{i-1})\,.
    \end{equation}
    \vspace{0.1cm}\\
Now by assumption $VE(\Omega_0)=0\,,$ therefore $\partial_y \Omega_0(x,x)=0\,.$ This implies that $S_{\Delta}(0)\to 0$ as the limit over all partitions is taken (we are implicitly using that a continuous function on a compact set is uniformly continuous), and therefore our definition agrees with that of the Riemann integral.
\subsection{The General Case}
Here we will prove the result for the general case of an $n$-dimensional compact manifold $X$ (with boundary). Note that we only need to prove that the result is true on closed rectangles in $\mathbb{R}^n\,,$ since the integral over the entire manifold is just a sum of the integrals over the nondegenerate $n$-simplices in the triangulation. 
\begin{proof}
The proof is essentially the same as in the one dimensional case. Consider a smooth function $f:[0,1]^n\to\mathbb{R}\,,$ where $[0,1]^n$ has coordinates given by $(x^1,\ldots, x^n)\,.$ Let $[0,1]^n_{\Delta}$ be a triangulation of $[0,1]^n$ and let $\Delta_{[0,1]^n}$ be the corresponding simplicial set.  We want to construct Riemann sums associated to the $n$-form $fdx^1\wedge\cdots\wedge dx^n\,.$ First, we antidifferentiate this to $\textup{Pair}([0,1]^n)\rightrightarrows [0,1]^n\,,$ with the $n$-cochain given by 
\\
\begin{equation}
    \Omega(x^1_0,\ldots,x_0^n,\ldots,x_n^1,\ldots x_n^n)=f(x_0^1,\ldots,x^n_0)\,\textup{Vol}_{\Delta}(x^1_0,\ldots,x_0^n,\ldots,x_n^1,\ldots x_n^n)\,.
\end{equation}
\\ 
This cochain is normalized and $S_n$-antisymmetric. Of course, taking the limit over all triangulations (using this cochain) give us the desired integral.
\vspace{0.5cm}\\Any other normalized, $S_n$-antisymmetric cochain $\Omega'$ such that $VE(\Omega')=fdx^1\wedge\cdots\wedge dx^n$ differs from $\Omega$ by some normalized, $S_n$-antisymmetric cochain $\Omega_0$ such that $VE(\Omega_0)=0\,.$ Let $\Omega_0$ be such a cochain. The only thing we need to verify is that
\begin{equation}
   \lim_{\Delta_{[0,1]^n}} \sum_{\{\Delta^n\}} \Omega_0(\Delta^n)=0\,,
\end{equation}
where the sum is over all nondegenerate $n$-simplices $\Delta^n$ in $[0,1]^n_{\Delta}\,,$ and where we are taking the limit over all triangulations. This follows from Taylor's theorem, \cref{asymptoticsn} and compactness, as in the one dimensional case.
\end{proof}
\begin{appendices}
\section{Lie Groupoids and Lie Algebroids, Nerves, Cochains and Forms}\label{groupoids}
Here we will briefly recall basic concepts relevant to the rest of the paper (see \cite{Crainic} for more details). In particular, there is a functor 
\begin{align}
&\mathbf{B}^{\bullet}:\textup{Lie groupoids}\to \textup{Simplicial manifolds}\,,
\;G\mapsto\mathbf{B}^{\bullet}G\,.
\end{align}
In this section we are going to describe this functor in the usual way, and we will describe the van Est map with respect to this description of $\mathbf{B}G\,.$
\\\\Note that we will not be using cohomology in this paper, although it will be mentioned briefly.
\begin{definition}
A (small) groupoid is a category $G\rightrightarrows X$ for which the objects $X$ and morphisms $G$ are sets and for which every morphism is invertible. A Lie groupoid is a (small) groupoid $G\rightrightarrows X$ such that $X\,, G$ are smooth manifolds such that the source and target maps, denoted $s\,,t$ respectively, are submersions, and such that all structure maps are smooth, ie.
\begin{align*}
    & i:X\to G
    \\ & m:G\sideset{_t}{_{s}}{\mathop{\times}} G\to G
    \\& \textup{inv}:G\to G
\end{align*}
are smooth (these maps are the identity, multiplication/composition and inversion, respectively). The embedding $i$ is called the identity bisection (since it is a section of both the source and target maps). A morphism between Lie groupoids $G\to H$ is a smooth functor between them.
\end{definition}
\vspace{0.1cm}
\begin{definition}\label{notation}
There is a functor 
\begin{equation}
  \mathbf{B}^{\bullet}:\textup{Lie groupoids}\to \textup{Simplicial manifolds}\,,\;G\mapsto\mathbf{B}^\bullet G\,,
\end{equation}
where $\mathbf{B}^0 G=X\,,\,\mathbf{B}^1 G=G\,,$ and 
\begin{equation}
    \mathbf{B}^n G=\underbrace{G\sideset{_t}{_{s}}{\mathop{\times}} G \sideset{_t}{_{s}}{\mathop{\times}} \cdots\sideset{_t}{_{s}}{\mathop{\times}} G}_{n \text{ times}}\,,
\end{equation}
which is called the space of 
$n$-composable arrows; $\mathbf{B}^\bullet G$ is called the nerve of $G\,.$
\end{definition}
\vspace{0.1cm}
\begin{definition}\label{cochain}
An $n$-cochain $\Omega$ (valued in an abelian Lie group $\mathcal{A}$) on a Lie groupoid $G\rightrightarrows X$ is a smooth function $\Omega:\mathbf{B}^{n}G\to \mathcal{A}\,.$ For this paper $\mathcal{A}$ will be $\mathbb{C}\,.$
\end{definition}
\vspace{0.1cm}
The infinitesimal counterpart of a Lie groupoid is a Lie algebroid.
\vspace{0.1cm}
\begin{definition}
A Lie algebroid is a triple ($\mathfrak{g},[\cdot,\cdot],\alpha)$ consisting of 
\begin{enumerate}
    \item A vector bundle $\pi:\mathfrak{g}\to X\,,$
    \item A vector bundle map (called the anchor map) $\alpha:\mathfrak{g}\to TX\,,$
    \item A Lie bracket $[\cdot,\cdot]$ on the space of sections $\Gamma(\mathfrak{g})$
\end{enumerate}
such that for all smooth functions $f\,,$ $[Y,fZ]=(\alpha(Y)f)Z+f[Y,Z]\,.$
\end{definition}
\vspace{0.1cm}
\begin{definition}
An $n$-form on a Lie algebroid is a section of $\bigwedge^n\mathfrak{g}^*\,.$
\end{definition}
Associated to any Lie groupoid $G\rightrightarrows X$ is a Lie algebroid $\mathfrak{g}\to X\,,$ whose underlying vector bundle is given by the normal bundle of $i(X)\subset G\,.$ This vector bundle can be identified with $i^*\textup{ker}(s_*)\,,$ ie. vectors tangent to the source fibers at the identity bisection. The anchor map is given by $i^*t_*$ and the Lie bracket is obtained in the same way it is obtained on Lie groups: by left translating sections in $\Gamma(\mathfrak{g})$ to vector fields in $TG$ tangent to the source fibers, and evaluating the Lie bracket of vector fields at the identity bisection. 
\\\\We can describe the left translation as follows: for $g\in G$ and $V_{t(g)}\in \mathfrak{g}$ a vector over $t(g)\,,$ there is a vector tangent to the source fiber at $g\,,$ induced by the multiplication map $g'\mapsto m(g,g')\,,$ where $s(g')=t(g)\,.$ 
\section{The van Est Map}\label{van est section}
Given $V\in\Gamma(\mathfrak{g})$ (ie. a section of $\mathfrak{g}\to X$), we can left translate it to a vector field $L_V$ on $G\,.$ Now suppose $\Omega$ is an $n$-cochain, we then get an $(n-1)$-cochain $L_V\Omega$ by defining
\vspace{0.05cm}\\
\begin{equation}
    (L_V\Omega)(g_1,\ldots,g_{n-1}):=L_V\big[\Omega(g_1,\ldots,g_{n-1},\cdot)\big]\vert_{t(g_{n-1})}\,,
\end{equation}$\,$
\vspace{0.05cm}\\
ie. we differentiate $\Omega$ in the final component and evaluate it at the identity $t(g_{n-1})\,.$ Now we are ready to define the van Est map, at the level of cochains.
\vspace{0.05cm}\\
\begin{definition}(see \cite{Crainic}, \cite{weinstein1})
The van Est map is a degree preserving map
\begin{equation}
    \textup{normalized cochains on }G\to\textup{forms on }\mathfrak{g}\,,
\end{equation}
defined as
\vspace{0.05cm}\\
\begin{align}
    VE(\Omega)(V_1,\ldots, V_n)=\sum_{\sigma\in S_n} \textup{sgn}(\sigma)L_{V_{\sigma(n)}}\cdots L_{V_{\sigma(1)}}\Omega\,.
\end{align} $\,$
\vspace{0.05cm}\\
Here $V_1,\ldots V_n\in \mathfrak{g}$ are vectors over the same point in $X\,.$ In order to make sense of this definition we need to choose extensions of $V_1,\ldots V_n$ to local sections of $\mathfrak{g}\to X,$ however the result is independent of the chosen extensions.\end{definition}
\vspace{0.2cm}
\begin{definition}\label{anti}
If $\textup{VE}(\Omega)=\omega$ then we say that $\Omega$ antidifferentiates $\omega\,.$ 
\end{definition}
\section{The Pair Groupoid}\label{pair section}
The only groupoid essential to \Cref{theorem} is the pair groupoid:
\begin{definition}
Let $X$ be a manifold. We define the pair groupoid, denoted $\textup{Pair}(X)\rightrightarrows X\,,$ to be the Lie groupoid whose objects are the points in $X$ and whose arrows are the points in $X\times X\,;$ an arrow $(x,y)$ has source and target $x,y\,,$ respectively. Composition is given by $(x,y)\cdot(y,z)=(x,z)\,.$
\end{definition}
The Lie algebroid of $\textup{Pair}(X)$ is the tangent bundle $TX$ (the anchor map provides the isomorphism).
\end{appendices}

\end{document}